\documentclass[12pt]{article}
\usepackage{amsmath,amssymb,amsthm}

\newtheorem{theorem}{Theorem}
\newtheorem{lemma}[theorem]{Lemma}
\newtheorem{corollary}[theorem]{Corollary}

\title{Finite groups determined by an inequality of the orders of their subgroups II}
\author{Marius T\u arn\u auceanu}
\date{October 25, 2016}

\begin{document}

\maketitle

\begin{abstract}
    In this note we study a class of finite groups for which the orders of subgroups satisfy a certain inequality.
    In particular, characterizations of the well-known groups $\mathbb{Z}_2\times\mathbb{Z}_2$ and $S_3$ are obtained.
\end{abstract}

{\small
\noindent
{\bf MSC2000\,:} Primary 20D60, 20D30; Secondary 20D15, 11A25.

\noindent
{\bf Key words\,:} finite groups, subgroup lattices, number of
subgroups, arithmetic functions.}

\section{Introduction}

Let $G$ be a finite group, $L(G)$ be the subgroup lattice of $G$ and
\[\sigma_1(G)=\sum_{H\in L(G)}\,\frac{|H|}{|G|}=\sum_{H\in L(G)}\,\frac{1}{|G:H|}\,.\]The starting point for our discussion is given by the paper \cite{1},
where the groups $G$ with $\sigma_1(G)\leq 2$ have been determined. Recall also several basic properties of the function $\sigma_1$:
\begin{itemize}
\item[-] if $G$ is cyclic of order $n$ and $\sigma(n)$ denotes the sum of all divisors of $n$, then $\sigma_1(G)=\frac{\sigma(n)}{n}$\,;
\item[-] $\sigma_1$ is multiplicative, i.e. if $G_i$, $i=1,2,\dots,m$, are finite groups of coprime orders, then $\sigma_1(\prod_{i=1}^m G_i)=\prod_{i=1}^m \sigma_1(G_i)$;
\item[-] $\sigma_1(G)\geq\sigma_1(G/H),\, \forall\, H\unlhd G$.
\end{itemize}

By refining the proof of Theorem 1 of \cite{1}, in the current paper we will determine the finite groups $G$ satisfying $\sigma_1(G)\leq 2+\frac{4}{|G|}$.
Our main result is the following.

\begin{theorem}\label{th:C1}
    Let $G$ be a finite group of order $n$. Then:
\begin{itemize}
\item[{\rm a)}] $\sigma_1(G)<2+\frac{4}{n}$ if and only if $G$ is cyclic and $\sigma(n)<2n+4$ or $G\cong\mathbb{Z}_2\times\mathbb{Z}_2$;
\item[{\rm b)}] $\sigma_1(G)=2+\frac{4}{n}$ if and only if $G$ is cyclic and $\sigma(n)=2n+4$ or $G\cong\mathbb{Z}_3\times\mathbb{Z}_3$ or $G\cong S_3$.
\end{itemize}
\end{theorem}

Two nice characterizations of $\mathbb{Z}_2\times\mathbb{Z}_2$ and $S_3$ can be inferred from the above theorem.

\begin{corollary}
    $\mathbb{Z}_2\times\mathbb{Z}_2$ is the unique non-cyclic group $G$ satisfying $\sigma_1(G)<2+\frac{4}{|G|}$\,, while $S_3$
    is the unique non-abelian group $G$ satisfying $\sigma_1(G)=2+\frac{4}{|G|}$\,.
\end{corollary}

By Theorem 1 of \cite{1}, a finite group $G$ with $\sigma_1(G)\leq 2$ is cyclic of deficient or perfect order. Also, in Lemma 4 below we will show
that a finite group $G$ with $\sigma_1(G)<2+\frac{4}{|G|}$ is always nilpotent. Inspired by these results, we came up with the following natural problem:
is there a constant $c\in(2,\infty)$ such that if $\sigma_1(G)<c$ then $G$ is nilpotent? The answer to this problem is negative, as shows our next theorem.

\begin{theorem}\label{th:C2}
    There are sequences of finite non-nilpotent groups $(G_n)_{n\in\mathbb{N}}$ such that $\sigma_1(G_n)\searrow 2$ for $n\rightarrow\infty$.
\end{theorem}

Finally, we note that an interesting open problem is whether there is a constant $c\in(2,\infty)$ such that if $\sigma_1(G)<c$ then $G$ is solvable.

Most of our notation is standard and will not be repeated
here. Basic definitions and results on groups can be found in
\cite{3}. For subgroup lattice concepts we refer the reader to
\cite{2} and \cite{4}.

\section{Proof of the main results}

We start by proving two auxiliary results.

\begin{lemma}
    Let $G$ be a finite group. If $\sigma_1(G)<2+\frac{4}{|G|}$ then $G$ is nilpotent, while if $\sigma_1(G)=2+\frac{4}{|G|}$ then $G$ is nilpotent or $G\cong S_3$.
\end{lemma}\newpage

\begin{proof}
    Assume that $\sigma_1(G)<2+\frac{4}{|G|}$ and $G$ is not nilpotent, that is it contains a non-normal maximal subgroup $M$. Then $M$ coincides with its normalizer in $G$ and so it has
    exactly $r=|G:M|$ conjugates, say $M_1, M_2,..., M_r$. On the other hand, $L(G)$ cannot consist only of $1$, $G$ and $M_1, M_2,..., M_r$. Therefore there is $N\leq G$ such that $N\neq 1, G, M_1, M_2,..., M_r$. Since
    \[2+\frac{4}{|G|}>\sigma_1(G)\geq\frac{1+r|M|+|N|+|G|}{|G|}=2+\frac{|N|+1}{|G|}\,,\]it follows that $N$ is a normal subgroup of order $2$ and $L(G)\!=\!\{1, G, M_1, M_2,...,\newline M_r, N\}$. Then either $N\subset M_i,\, \forall\, i=1,2,...,r$, or $N\cap M_i=1,\, \forall\, i=1,2,...,r$. We infer that $G$ is either a $2$-group or a cyclic group of order $2p$ for some odd prime $p$, a contradiction.

    Assume now that $\sigma_1(G)=2+\frac{4}{|G|}$ and $G$ is not nilpotent. Then, under the above notation, we must have $|N|=3$ and thus $G$ is a group of order $3q$ for some prime $q$. Clearly, the conditions $r\mid 3$ and $r\equiv 1 \,({\rm mod}\, q)$ imply $r=3$ and $q=2$. Hence $G\cong S_3$, completing the proof.
\end{proof}

\begin{lemma}
    Let $G$ be a non-cyclic $p$-group. If $\sigma_1(G)<2+\frac{4}{|G|}$ then $G\cong\mathbb{Z}_2\times\mathbb{Z}_2$, while if $\sigma_1(G)=2+\frac{4}{|G|}$ then $G\cong\mathbb{Z}_3\times\mathbb{Z}_3$.
\end{lemma}

\begin{proof}
    Let $|G|=p^n$ and $G/\Phi(G)\cong\mathbb{Z}_p^k$. Then $2\leq k\leq n$. Assume that $\sigma_1(G)\leq 2+\frac{4}{p^n}$ and denote by $a_{k,p}(i)$ the number of subgroups of order $p^i$ in $\mathbb{Z}_p^k$, $i=0,1,...,k$. If $k\geq 3$ then
    \begin{align*}
        \sigma_1(G)
        &\geq\frac{1}{p^n}\left(a_{k,p}(0)p^{n-k}+a_{k,p}(1)p^{n-k+1}+a_{k,p}(k-1)p^{n-1}+a_{k,p}(k)p^n\right)\\
        &= \frac{1}{p^k}\left(a_{k,p}(0)+a_{k,p}(1)p+a_{k,p}(k-1)p^{k-1}+a_{k,p}(k)p^k\right) \\
        &= \frac{1}{p^k}\left(1+\frac{p^k-1}{p-1}p+\frac{p^k-1}{p-1}p^{k-1}+p^k\right) \\
        &>\frac{1}{p^k}\left(1+4p^k\right)>4>2+\frac{4}{p^n}\,,
    \end{align*}contradicting our hypothesis. Consequently, $k=2$. For $n>2$ we infer that
    \[ 2+\frac{4}{p^n}\geq\sigma_1(G)\geq\frac{1+(p+1)p^{n-1}+p^{n-2}+p^n}{p^n}\,,\]which leads to
    \[ 4\geq p^{n-1}+p^{n-2}+1,\]a contradiction. Therefore $n=2$, i.e. $G\cong\mathbb{Z}_p\times\mathbb{Z}_p$. Moreover, we have
    \[ 2+\frac{4}{p^2}\geq\sigma_1(G)=\frac{1+(p+1)p+p^2}{p^2}\Longleftrightarrow 4\geq p+1\Longleftrightarrow p\in\{2,3\}.\]Thus $G\cong\mathbb{Z}_2\times\mathbb{Z}_2$ for
    $\sigma_1(G)<2+\frac{4}{p^2}$ and $G\cong\mathbb{Z}_3\times\mathbb{Z}_3$ for $\sigma_1(G)=2+\frac{4}{p^2}$\,, as desired.
\end{proof}

We are now able to prove our main results.

\begin{proof}[Proof of Theorem \ref{th:C1}]
\begin{itemize}
\item[{\rm a)}] Assume that $\sigma_1(G)<2+\frac{4}{n}$. Then $G$ is nilpotent by Lemma 4, which implies that it can be written as a direct product of its Sylow $p_i$-subgroups:
    \[ G\cong\prod_{i=1}^m G_i.\tag{$*$}\]For every $i=1,2,...,m$, we have
    \[ \sigma_1(G_i)\leq\sigma_1(G)<2+\frac{4}{n}\leq 2+\frac{4}{|G_i|}\,.\]By Lemma 5 it follows that either $G_i$ is cyclic or $G_i\cong\mathbb{Z}_2\times\mathbb{Z}_2$. We infer that either $G$ is cyclic or $G\cong\mathbb{Z}_2\times\mathbb{Z}_2\times\mathbb{Z}_{n'}$ for some odd positive integer $n'$. Suppose that $G$ is not cyclic. If $n'>1$ we obtain
    \[ 2+\frac{1}{n'}=2+\frac{4}{n}>\sigma_1(G)=\sigma_1(\mathbb{Z}_2\times\mathbb{Z}_2)\sigma_1(\mathbb{Z}_{n'})=\frac{11\sigma(n')}{4n'}\,,\]or equivalently
    \[ 11\sigma(n')<8n'+4\,,\]a contradiction. Thus $n'=1$ and $G\cong\mathbb{Z}_2\times\mathbb{Z}_2$.
\item[{\rm b)}] Assume that $\sigma_1(G)=2+\frac{4}{n}$ and $G\not\cong S_3$. Then $G$ is nilpotent by Lemma 4 and so it has a direct decomposition of type $(*)$. For $m\geq 2$ we have \[ \sigma_1(G_i)\leq\sigma_1(G)=2+\frac{4}{n}< 2+\frac{4}{|G_i|}\,,\, i=1,2,...,m,\]and again every $G_i$ is cyclic or isomorphic to $\mathbb{Z}_2\times\mathbb{Z}_2$ by Lemma 5. Therefore either $G$ is cyclic or $G\cong\mathbb{Z}_2\times\mathbb{Z}_2\times\mathbb{Z}_{n'}$ for some odd positive integer $n'$. In the second case the condition  $\sigma_1(G)=2+\frac{4}{n}$ leads to \[ 11\sigma(n')=8n'+4\,,\]a contradiction. For $m=1$ it follows that $G$ is a $p$-group, and consequently it is cyclic or isomorphic to $\mathbb{Z}_3\times\mathbb{Z}_3$ by Lemma 5. This completes the proof. \qedhere
\end{itemize}
\end{proof}

\begin{proof}[Proof of Theorem \ref{th:C2}]
    Let $(p_n)_{n\in\mathbb{N}}$ be the sequence of prime numbers. By Dirichlet's theorem we infer that for every
    $n\in\mathbb{N}$ there is a prime $q_n$ such that $p_n\mid q_n-1$. Let $G_n$ be the non-nilpotent group of order $p_nq_n$. This contains one subgroup of order $1$, $q_n$ subgroups of order $p_n$, one subgroup of order $q_n$, and one subgroup of order $p_nq_n$. Then
    \[\sigma_1(G_n)=\frac{1+p_nq_n+q_n+p_nq_n}{p_nq_n}=2+\frac{1+\frac{1}{q_n}}{p_n}\,,\]and clearly $\sigma_1(G_n)\searrow 2$ for $n\rightarrow\infty$, as desired.
\end{proof}

\vspace*{5ex}\small

\hfill
\begin{minipage}[t]{5cm}
Marius T\u arn\u auceanu \\
Faculty of  Mathematics \\
``Al.I. Cuza'' University \\
Ia\c si, Romania \\
e-mail: {\tt tarnauc@uaic.ro}
\end{minipage}

\end{document}